\documentclass[11pt]{amsart}

\usepackage{ae,aecompl}
\usepackage{esint}
\usepackage{graphicx}
\usepackage[all,cmtip]{xy}
\usepackage{amsmath, amscd}
\usepackage{booktabs}
\usepackage{amsthm}
\usepackage{amssymb}
\usepackage{amsfonts}
\usepackage{qsymbols}
\usepackage{latexsym}
\usepackage{mathrsfs}
\usepackage{cite}
\usepackage{color}
\usepackage{url}
\usepackage{enumerate}
\usepackage[english]{babel}

\usepackage{verbatim}
\usepackage[draft=false, colorlinks=true]{hyperref}

\allowdisplaybreaks

\newcommand {\C}{\mathbb C}

\newcommand {\R}{\mathbb R}

\newcommand{\Ha}{\mathbb{H}^n}

\newcommand {\vanish}[1]{\relax}

\newtheorem{theorem}{Theorem}[section]
\newtheorem{lemma}[theorem]{Lemma}
\newtheorem{prop}[theorem]{Proposition}

\theoremstyle{definition}

\numberwithin{equation}{section}
\protected\def\ignorethis#1\endignorethis{}
\let\endignorethis\relax

\title[Dimension free bounds]{Dimension free estimates for the vector-valued Hardy--Littlewood maximal function on the Heisenberg group}

\author[Ganguly]{Pritam Ganguly}
	\address[Pritam Ganguly]{ Institut f\"ur Mathematik \\
Universit\"at Paderborn \\ Warburger Str. 100,
33098 Paderborn, Germany.}\email{pritam1995.pg@gmail.com/ pritamg@math.uni-paderborn.de}

\author[Ghosh]{Abhishek Ghosh}
\address[Abhishek Ghosh]{Department of Mathematics, Indian Institute of Technology Jammu, Jagti, J \& K, 181221, India.}
\email{abhishek.ghosh@iitjammu.ac.in/ abhi170791@gmail.com}

\keywords{Dimension-free estimates, Hardy-Littlewood maximal function, Heisenberg group}
	\subjclass[2020]{Primary: 43A80, 42B25. Secondary: 22E25, 42B35}

\begin{document}

\begin{abstract}
In this article, we establish dimension-free Fefferman-Stein inequalities for the Hardy-Littlewood maximal function associated with averages over Kor\'anyi balls in the Heisenberg group. We also generalize the result to more general UMD lattices.  As a key stepping stone, we establish the $L^p$- boundedness of the vector-valued Nevo-Thangavelu spherical maximal function, which plays a crucial role in our proofs of the main theorems.
\end{abstract}
	
\maketitle	


\section{Introduction}
A fundamental pursuit in modern real-variable harmonic analysis is the development of dimension-free estimates for various operators. This line of research was motivated by the challenge of controlling the growth of the norm of the classical Hardy--Littlewood maximal operator $M_{\mathbb{R}^n}$, defined for a locally integrable function $f$ on $\mathbb{R}^n $ by  
\[
M_{\mathbb{R}^n}f(x) = \sup_{r > 0} \frac{1}{|B(x, r)|} \int_{B(x, r)} |f(y)| \, dy,
\]  
where $B(x, r)$ denotes the Euclidean ball centered at $ x $ with radius $r $. It is well-known that arguments involving traditional covering lemmas typically result in exponential growth of the $L^p$-operator norm of $M_{\mathbb{R}^n} $ with the dimension $n$.  This prompted the search for estimates whose bounds are independent of the dimension.  

Stein \cite{SteinSph} was the first to announce dimension-free bounds for the maximal operator $M_{\mathbb{R}^n}$, and a complete, detailed proof was later provided by Stein and Str\"omberg \cite{SS} in 1982. More precisely, they established that $$ \|M_{\mathbb{R}^n}\|_{L^p \to L^p} \leq C (p)\quad\quad (1<p\leq \infty) $$ where the constant $C (p)$ depends only on $ p $ and is independent of the dimension $ n $. 

This foundational work paved the way for broader developments in harmonic analysis, motivating the exploration of dimension-free estimates in more general settings. These include maximal functions where Euclidean balls are replaced by general convex sets, as well as other fundamental operators such as Calder\'on--Zygmund operators. Of course, the study of dimension-free bounds for such operators introduces new challenges, but the pursuit of understanding the underlying geometric and analytic structures in high-dimensional settings has driven research toward establishing such bounds. This topic has gathered significant attention over the past few years, with notable contributions from Bourgain \cite{Bourgain87} regarding convex symmetric bodies in $\mathbb{R}^n$ and recently in \cite{Bourgain-GAFA} the authors obtained dimension-free estimates for variational estimates of maximal functions over symmetric convex bodies. For their discrete couterparts we refer the article \cite{Bourgain-AJM}. For a comprehensive overview of recent developments, see the surveys \cite{Bourgain-survey, DGM}.

The point of departure and primary motivation  of this article lie in the realm of generalising classical results to the vector-valued setting. Extending the theory of singular integrals and maximal functions to the vector-valued setting is a vast and important program in real-variable harmonic analysis, laid by the foundational works of Benedek--Calder\'on--Panzone \cite{BCP}, Rivière \cite{Riviere}, Fefferman--Stein \cite{FS} and many others. A fundamental result in this direction, known as the Fefferman--Stein inequality, states that the inequality
\[
\biggl\|\Bigl(\sum_{j=1}^{\infty}|M_{\mathbb{R}^n}f_j(\cdot)|^q\Bigr)^\frac{1}{q}\biggr\|_{L^p(\mathbb R^n)}\leq C(n, p, q)\, \biggl\|\Bigl(\sum_{j=1}^{\infty}|f_j(\cdot)|^q\Bigr)^\frac{1}{q}\biggr\|_{L^p(\mathbb R^n)}
\]
 holds true for all $1<p, q<\infty.$ 
See \cite{FS} for other end-point results. This setting introduces additional complexities, as it requires uniform control over operator norms acting on vector-valued functions. Recently, in \cite{DK}, building upon the work of \cite{SS}, the authors proved the following dimension-free estimate:
\begin{theorem}
    \label{DK-dimfree-Rn}
    Let $1<p, q<\infty.$ Then there exists a constant $C(p, q),$ independent of dimension, such that 
\[
\biggl\|\Bigl(\sum_{j=1}^{\infty}|M_{\mathbb{R}^n}f_j(\cdot)|^q\Bigr)^\frac{1}{q}\biggr\|_{L^p(\mathbb R^n)}\leq C(p, q)\, \biggl\|\Bigl(\sum_{j=1}^{\infty}|f_j(\cdot)|^q\Bigr)^\frac{1}{q}\biggr\|_{L^p(\mathbb R^n)}
\]
for any  sequence $(f_j)_{j\geq 1}$  of measurable functions defined on $\mathbb{R}^n$ such that $\left(\sum_{j=1}^\infty|f_j(\cdot)|^q\right)^{\frac1q}\in L^p(\mathbb{R}^n).$
\end{theorem}
In the same paper \cite{DK}, Deleaval and Kriegler extend the aforementioned theorem by replacing the classical $\ell^q$ space with a general UMD Banach lattice. Furthermore, they establish an analogous result on $\mathbb{R}^{n+1}$, where the standard Euclidean distance is substituted with the Carnot-Carathéodory and Kor\'anyi distances associated with the framework of Grushin operators. See Theorem 4 of \cite{DK}. However, the analogue of Theorem \ref{DK-dimfree-Rn} in the context of the Heisenberg group--a close non-commutative relative of Euclidean space--left as an intriguing open problem in \cite{DK}.

The primary objective of this article is to investigate dimension-free estimate for the vector-valued  the Hardy-Littlewood maximal function on the Heisenberg group.  Notably, the work by \cite{DK} has been a significant source of motivation for our study. To provide context and clearly state our main results, we introduce certain notations at this stage, though detailed explanations are deferred to Section 2. Let $ \mathbb{H}^n:= \mathbb{C}^n \times \mathbb{R} $ denote the $(2n + 1)$-dimensional Heisenberg group, equipped with the group law defined for $ x, y \in \mathbb{H}^n $ by  
\[
x \cdot y := \left( \underline{x} + \underline{y},~ \overline{x} + \overline{y} + \frac{1}{2} \, \underline{x}^t J \underline{y} \right),
\]  
where $ J $ denotes the standard $ (2n \times 2n) $-symplectic matrix, that is, $J=\begin{bmatrix}0& -Id\\
Id&0\end{bmatrix}$.  Here and throughout this article, we adopt the convention of denoting the non-central component of $ x \in \mathbb{H}^n $ by $ \underline{x} \in \mathbb{C}^n $ and the central component by $ \overline{x} \in \mathbb{R} $. The Haar measure on $\mathbb{H}^n$ coincides with the Lebesgue measure $dx=d\underline{x}\,d\overline{x}$. We can equip $\Ha$ with a left-invariant metric induced by
the Kor\'anyi norm  defined by
 \begin{equation*}
     |x|:= \left( \|\underline{x} \|^4 + \overline{x}^2 \right)^{\frac{1}{4}} \quad\quad (x=(\underline{x}, \overline{x}) \in \mathbb{H}^n),
 \end{equation*}
where $\| \underline{x}  \|$ denotes the Euclidean norm of $\underline{x} \in \mathbb{C}^{n}$. In fact, this makes $\Ha$ a space of homogeneous type.  

 For $a\in \Ha$,  we denote the ball of radius $r>0$ and centered at $a$ by  $B(a, r):=\{x\in \Ha: |a^{-1} x|<r\},$ moreover, one has $|B(a,r)| = C_{Q} \, r^{Q},$ where $Q=(2n + 2)$ is known as the \textit{homogeneous dimension} of $\Ha$. Consequently, for $f\in L^1_{loc}(\Ha)$, the (centered) Hardy-Littlewood maximal function $\mathcal{M}f$ is defined by 
\begin{align}
\mathcal{M}f(x):=\sup_{r>0}\frac{1}{|B(x, r)|}\int_{B(x, r)} |f(y)|\, dy\quad\quad(x\in \Ha),   
\end{align}
where $|B(x, r)|$ denotes the measure of the ball. The standard $L^p$ boundedness for $p > 1$ and the weak $(1,1)$ boundedness of $\mathcal{M}$ follow from the more general theory established for spaces of homogeneous type, where an appropriate adaptation of the covering lemma is employed. For details, see Stein \cite{SteinBook}. Additionally, Li \cite{HQLi-Heisenberg} extended the work of Stein and Strömberg \cite{SS} to the Heisenberg group. By establishing a connection between ball averages and the Poisson semigroup and utilizing the Hopf-Dunford-Schwartz maximal ergodic theorem, Li derived weak $(1,1)$ estimates with explicit bounds that depend linearly on the dimension. Subsequently, building on the framework of \cite{SS}, Zienkiewicz \cite{Z} resolved the problem of dimension-free bounds for the $L^p$ operator norm of $\mathcal{M}$. This result addressed a question initially raised by Cowling, as noted in \cite{Z}.  On the vector-valued side, the Fefferman-Stein inequality for the Hardy-Littlewood maximal operator was established by Grafakos et al. \cite{Gra-Scand} in the broader setting of spaces of homogeneous type. Interestingly, through a linearization technique, they derived this result from vector-valued estimates for singular integral operators, which were also proved in the same paper. However, their approach was not designed to yield dimension-independent bounds. We shall state that theorem in the context of the Heisenberg group in the next section, as it will be useful for our purposes. As our first main result, we now present the following dimension-free estimate for the vector-valued maximal function on $\Ha$. 
\begin{theorem}
\label{Main-thm}
Let $1<p, q<\infty.$ Then there exist a constant $C(p, q),$ independent of dimension, such that
\[
\biggl\|\Bigl(\sum_{j=1}^{\infty}|\mathcal{M}f_j(\cdot)|^q\Bigr)^\frac{1}{q}\biggr\|_{L^p(\Ha)}\leq C(p, q)\, \biggl\|\Bigl(\sum_{j=1}^{\infty}|f_j(\cdot)|^q\Bigr)^\frac{1}{q}\biggr\|_{L^p(\Ha)},
\]
for any  sequence $(f_j)_{j\geq 1}$  of measurable functions defined on $\mathbb{H}^n$ such that $\left(\sum_{j=1}^\infty|f_j(\cdot)|^q\right)^{\frac1q}\in L^p(\mathbb{H}^n).$
\end{theorem}

Next, following the approach in \cite{DK}, we extend the above result by considering more general UMD Banach lattices in place of $\ell^q$. To proceed, we first introduce some necessary notations. In what follows, $Z$ will denote a Banach lattice of measurable functions in a $\sigma-$finite measure space $(\Omega, d \omega),$ and $|\cdot|$ will denote the absolute value in $Z,$ i.e., $|z|:=\sup\{z, -z\},$ which canonically extend as $|z|(\omega)=|z(\omega)|, \,\omega\in \Omega.$  Also, in the sequel $L^p(\Ha, Z)$ will denote the Bochner-space of measurable functions $f:\Ha\to Z$ satisfying $\int_{\Ha}\|f(x)\|_{Z}\, dx<\infty.$ A Banach lattice $Z$ is a UMD-lattice if the Hilbert transform 
$$Hf(s)=\text{p.v.}\,\int_{\mathbb{R}}\frac{f(t)}{s-t}\, dt,\quad\,\, f\in \mathcal{S}(\mathbb{R}),$$
extends boundedly from $L^p(\mathbb{R}, Z)$ to itself for $1<p<\infty.$ Further, $L^p(\Ha)\otimes Z$ denotes the dense subspace
\begin{align}
\big\{f(x, \omega):=\sum_{j=1}^{N}f_{j}(x)\, z_{j}(\omega): f_{j}\in L^p(\Ha), z_{j}\in Z, N\in \mathbb{N}\big\}
\end{align}
of $L^p(\Ha, Z).$ For $f\in L^p(\Ha)\otimes Z,$ the UMD lattice valued Hardy-Littlewood maximal operator on the Heisenberg group is defined as
\begin{align}
&\mathcal{M}(f)(x, \omega)=\sup_{r>0}\frac{1}{|B(x, r)|}\int_{B(x, r)} |f(y, \omega)|\, dy.
\end{align}
Now we state our result.
\begin{theorem}
\label{main-thm2}
Let $1<p<\infty$ and $Z=Z(\Omega,\, d\omega)$ be a UMD lattice. Then there exists a constant $C(p, Z),$ independent of $n,$ such that
$$\bigg\| \big\|\mathcal{M}(f)(\cdot, z)(x)\big\|_{Z}\bigg\|_{L^p(\Ha)}\leq C(p, Z)\, \|f\|_{L^p(\Ha,\, Z)}.$$
\end{theorem}

We now outline the key ideas underlying the proofs, which will motivate the development of another vector-valued estimate later in this article. To set the stage, we first return to the Euclidean setting, where the approach for obtaining dimension-free estimates for the scalar-valued maximal function originates from the seminal work of Stein and Strömberg \cite{SS}. Their method is based on two fundamental steps. First, we decompose $\mathbb{R}^n$ as $\mathbb{R}^{n-k}\times \mathbb{R}^k,$ and represent any point $y\in \mathbb{R}^n$ as $y=(y_1, y_2)\in \mathbb{R}^{n-k}\times \mathbb{R}^k.$  The crucial initial step is to express the Hardy-Littlewood maximal function in terms of averages over lower-dimensional weighted maximal functions, denoted schematically as 
\begin{align}
\label{SS1}
M_{\mathbb{R}^n}f(x)\leq \int_{\mathcal{O}(n)} M_{k}^{A}f(x)\, dA,   
\end{align}
where $dA$ denotes the normalized Haar measure on the orthonormal group of matrices $\mathcal{O}(n),$ and
$$M_{k}^{A}f(x):=\sup_{r>0}\frac{\int_{|y_1|<r} f(x-A(y_{1}, 0))\,|y_1|^{k} \, dy_{1}}{\int_{|y_1|<r} |y_1|^{k}\, dy_{1}},$$
for all $1\leq k\leq n.$ Secondly, Stein and Strömberg \cite{SS} established the following estimate:
\begin{align}
\label{SS2}
\|M_{k}^{A}f\|_{L^p(\mathbb{R}^n)}\leq C(n-k, p)\|f\|_{L^p(\mathbb{R}^n)}
\end{align}
 when $n\geq k+ 3$ and $p>\frac{n-k}{n-k-1}.$ The proof of this result relies critically on the boundedness of Stein's spherical maximal function, established in \cite{SteinSph}.  The dimension-free estimates follow directly from these two steps. Specifically, for $1 < p \leq \infty$ with $n \leq \frac{p}{p-1}$, one can employ standard techniques to establish the boundedness of $M_{\R^n}$. For the more intricate case when $n > \frac{p}{p-1}$, we decompose $\mathbb{R}^n = \mathbb{R}^{n-k} \times \mathbb{R}^k$, where  
 \[
n-k := \left\lfloor \max\left\{2, \frac{p}{p-1}\right\} \right\rfloor + 1.  
\]  
 The dimension-free estimate then follows by applying \eqref{SS1} and \eqref{SS2}, taking into account that the constant $C(n-k, p)$ simplifies to depend only on $p$.

 Following this approach, the authors in \cite{DK} extended the method to the vector-valued setting. A crucial step in their analysis was establishing vector-valued $L^p$ estimates for the spherical maximal function on $\R^n.$ Then, in the spirit of Rubio de Francia's approach, they dominated the maximal function by a series of frequency-localized components. By applying explicit estimates of the Bessel function, they estimated these components, and through complex interpolation, they obtained the desired result.

To establish the main results of this article, we follow the aforementioned approach. As the reader might expect, a key role will be played by the Heisenberg counterpart of Stein's spherical maximal function, known as the Nevo--Thangavelu spherical maximal function. We recall its definition below: Let $\mu$ be the normalised surface measure on the horizontal sphere $\mathbb{S}^{2n-1} \times \{0\} \subseteq \Ha$. For $r>0,$ the dilate of $\mu,$ $\mu_{r}$ is defined as $\langle f, \mu_{r}\rangle=\int f(rx, 0)\, d\mu,$ for Schwartz class functions. The spherical means taken over horizontal spheres are defined by  
\begin{equation}\label{eq: horizontal spherical av}
    A_{r} f(x)=f*\mu_{r}(x) = \int_{\mathbb{S}^{2n-1}} f\big(\underline{x}-r \underline{u}, \overline{x}- {\textstyle\frac{r}{2} }   \, \underline{x}^t J \underline{u}\big) \,d\mu(\underline{u}), 
\end{equation}
for $r>0$, and $x\in \Ha$. The Nevo--Thangavelu spherical maximal function is defined (initially on Schwartz class functions) as $$M_{S}f(x)=\sup_{r>0}|A_{r}f(x)|\quad\quad(x\in \Ha).$$

As the name suggests, the Nevo--Thangavelu spherical maximal function was introduced by Nevo and Thangavelu in \cite{NeT}. They proved that $M_{S}$ is bounded on $L^p(\mathbb{H}^n)$ for $p > \frac{2n-1}{2n-2}$.
Subsequently, it was independently proved by Narayanan and Thangavelu in \cite{NaT}, and by Müller and Seeger in \cite{MS}, that $M_{S}$ is bounded on $L^p(\mathbb{H}^n)$ if and only if $p > \frac{2n}{2n-1}$. In recent years, the operator $M_{S}$ has been explored in various contexts by several authors; see \cite{BHRT, Roos} and the references therein for further details.

We remark in passing that, in the context of $\mathbb{H}^n$, there is another type of spherical maximal function associated with averages over spheres defined in terms of the Korányi norm, first studied by Cowling \cite{Cowling}. Notably, Kor\'anyi spheres are codimension-one surfaces, whereas horizontal spheres are codimension-two in $\mathbb{H}^n$. This maximal function has also attracted considerable attention in various contexts; see, for instance, \cite{GT, Raj}. However, unlike the Korányi spherical means,  the Nevo–Thangavelu spherical means are linked to orbital integrals arising from the action of $U(n)$ on $\mathbb{H}^n$, a connection that is crucial for our proofs. This is the reason why the Kor\'anyi spherical maximal function is not a useful tool in our analysis.

 Returning to $M_S$, we establish the following vector-valued estimates for the Nevo–Thangavelu maximal operator $M_S$:

\begin{theorem}
\label{Sph-Main-Heisenber}
Let $n>1$, and $(f_j)_{j\geq1}$ be a sequence of measurable functions defined on $\Ha$. If $\bigl(\sum_{j=1}^{\infty}|f_j(\cdot)|^q\bigr)^\frac{1}{q} \in L^p(\Ha)$, then 
\[
\biggl\|\Bigl(\sum_{j=1}^{\infty}|M_{S}f_j(\cdot)|^q\Bigr)^\frac{1}{q}\biggr\|_{L^p(\Ha)}\leq C(n, p, q)\biggl\|\Bigl(\sum_{j=1}^{\infty}|f_j(\cdot)|^q\Bigr)^\frac{1}{q}\biggr\|_{L^p(\Ha)},
\]
provided $\frac{Q}{Q-2}<p, q<\frac{Q}{2},$ where $Q=2n+2$ is the homogeneous dimension of $\Ha.$
\end{theorem}
Theorem~\ref{Sph-Main-Heisenber} is quintessential in the proof of our main result Theorem~\ref{Main-thm}, however, we believe it is also of independent interest. To prove this theorem, the primary challenge when extending the techniques from \cite{DK} to the Heisenberg group is that the Fourier transform becomes operator-valued. Similar to the Euclidean case, the horizontal spherical mean can be interpreted as a radial Fourier multiplier on the Heisenberg group, and it is possible to decompose the associated maximal function into a sum of frequency-localized components using the spectral theory of scaled Hermite operators. However, at this stage, it is unclear how to derive the necessary estimates for these components to obtain the $L^p$-estimate. To circumvent this issue, we move away from the Fourier transform approach based on spectral theory of the sublaplacian. Instead treating the Heisenberg group $\Ha$ as $\R^{2n+1}$, by using a M\"uller-Seeger type decomposition of the spherical mean with oscillatory integrals, we are able to address this difficulty. We provide a detailed explanation of this method in the next section.

Once we have the Theorem~\ref{Sph-Main-Heisenber} at our disposal, the proof of Theorem~\ref{Main-thm} follows from the fact that the Hardy-Littlewood maximal function $\mathcal{M}$ can be dominated by the composition of the Hardy-Littlewood maximal function on the vertical direction together with $M_{S}.$ These steps are elaborated in Section~\ref{Main-Sec}.

Throughout this article, we use the shorthand $C(A)$ to denote constants that depend only on the parameter $A$. These constants may vary from line to line without explicit mention.

\section{Preliminaries}
\label{Prelim}
In this section, we gather the necessary background on the Heisenberg group. General reference for the basics of Heisenberg groups is the monograph of Thangavelu \cite{TH}. See also the book of Folland \cite{Folland}.  However, we shall use slightly different notations for our convenience. 
\subsection{Structure of the Heisenberg group}	Recall that the $(2n+1)$- dimensional Heisenberg group $\mathbb{H}^n:=\mathbb{C}^n\times\mathbb{R}$ is a step two nilpotent Lie group where the Lebesgue measure $dx=d\underline{x}d\overline{x}$ on $\mathbb{C}^n\times\mathbb{R}$ serves as the Haar measure. We have a family of non-isotropic dilations defined by $\delta_{r}(\underline{x}, \overline{x}):=(r \underline{x}, r^2 \overline{x})$, for all $(\underline{x}, \overline{x}) \in \mathbb{H}^n$, for every $r>0$ and the Kor\'anyi norm is homogeneous of degree 1 with respect to this family of dilations, that is, $|\delta_{r}(\underline{x}, \overline{x})|= r \,  |(\underline{x}, \overline{x})|$. The convolution of $f$ with $g$ on $\Ha$ is defined by
\begin{equation*}
   f * g \, (x) = \int_{\Ha}  f(x y^{-1}) g(y) dy, \ \ \ x \in \Ha.
\end{equation*}
    We let $ \mathfrak{h}_n $ stand for the Heisenberg Lie algebra consisting of left invariant vector fields on $ \mathbb{H}^n .$  A  basis for $ \mathfrak{h}_n $ is provided by the $ 2n+1 $ vector fields
	 $$ X_j = \frac{\partial}{\partial{x_j}}+\frac{1}{2} x_{n+j} \frac{\partial}{\partial \overline{x}}, \,\,X_{n+j} = \frac{\partial}{\partial{x_{n+j}}}-\frac{1}{2} x_j \frac{\partial}{\partial \overline{x}}, \,\, j = 1,2,..., n $$
	 and $ T = \frac{\partial}{\partial \overline{x}}.$  These correspond to certain one parameter subgroups of $ \mathbb{H}^n.$ The sublaplacian on $\Ha$ is defined by $\mathcal{L}:=-\sum_{j=1}^{2n}X_j^2 $ which is given explicitly by
	 $$\mathcal{L}=-\Delta_{\C^n}-\frac{1}{4}|\underline{x}|^2\frac{\partial^2}{\partial \overline{x}^2}+\left(\sum_{j=1}^{n}\left(x_j\frac{\partial}{\partial x_{n+j}}-x_{n+j}\frac{\partial}{\partial x_j}\right)\right)\frac{\partial}{\partial \overline{x}}$$ where  $\Delta_{\C^n}$ stands for the Laplacian on $\C^n.$ 
	 This is a sub-elliptic operator and homogeneous of degree $2$ with respect to the parabolic dilation given by $\delta_{r}(\underline{x}, \overline{x})=(r \underline{x}, r^2 \overline{x}).$ The sublaplacian is also invariant under  rotation i.e., $$R_{\sigma}\circ \mathcal{L}=\mathcal{L}\circ R_{\sigma},~\sigma\in U(n),$$
     where $R_\sigma$ is the horizontal rotation by $\sigma$,  defined as $R_\sigma f(x):=f(\sigma \cdot \underline{x}, \bar{x}),~x\in \Ha.$

We close this subsection with the Fefferman-Stein inequality for the Hardy-Littlewood maximal operator $\mathcal{M}$ on  $\Ha.$ 

\begin{theorem}\cite[Theorem 1.2]{Gra-Scand}
\label{ThmFS}
Let $1<p, q<\infty.$ Then there exists a constant $C(n, p, q)$ such that 
\[
\biggl\|\Bigl(\sum_{j=1}^{\infty}|\mathcal{M} f_j(\cdot)|^q\Bigr)^\frac{1}{q}\biggr\|_{L^p(\Ha)}\leq C(n, p, q)\, \biggl\|\Bigl(\sum_{j=1}^{\infty}|f_j(\cdot)|^q\Bigr)^\frac{1}{q}\biggr\|_{L^p(\Ha)}.
\]
 holds true.
    
\end{theorem}

Next, considering the Heisenberg group $\mathbb{H}^n$ as $\mathbb{R}^{2n+1}$, we present a decomposition of the horizontal spherical means using oscillatory integrals. This decomposition, originally introduced by M\"uller and Seeger \cite{MS}, provides a refined analytical framework for studying the behavior of the associated maximal function.
\subsection{M\"uller-Seeger decomposition of spherical means}
We  consider the spherical means of  a function $f$ on $\mathbb{H}^n$ defined by 
\begin{equation}
	\label{eq:defin}
	f\ast \mu_r(x)=\int_{S^{2n-1}}f\Big(\underline{x}-r\,\underline{u},\overline{x}-\frac{r}2 \, \underline{x}^t J \underline{u}\Big)\,d\mu(\underline{u}) \quad\quad (x\in \Ha)
\end{equation}
where $\mu$ is the normalised surface measure on the complex sphere $S^{2n-1}\times \{0\}$ in $\mathbb{H}^n$.
Parameterize the sphere $S^{2n-1}$ near the north pole $e_{2n},$ and the parameterization is $(\omega', \phi(\omega'))$ where $\omega'\in B^{2n-1}(\vec{0}, c_{n})$ and the function $\phi(\omega')=\sqrt{1-|\omega'|^2}.$ Therefore, using this parametrization, the measure $\mu$ can be thought of as the pairing of the distribution  \[\gamma_1(x') \gamma_2(x_{2n}, \bar{x})  \delta(x_{2n}-\phi(x'), \bar{x}),\] where 
$\gamma_{1}, \gamma_{1}$ are infinitely smooth functions supported in small balls $B^{2n-1}(0, \rho)$ and $B^{2}(e, \rho^2),$ respectively, where $e=(1, 0)\in \mathbb{R}^{2},$ and $\delta$ is the Dirac measure  in $\mathbb{R}^{2}.$ Moreover, $\rho$ is chosen sufficiently small such that we have $g(0)=1, \nabla g(0)=1, D^2g(0)=I_{2n-1}, g^{(3)}(0)=0.$

Using the Fourier inversion formula for Dirac measures, we can express the measure $\mu$ as follows:
\begin{equation}
    \label{diracFT}
    \mu(x)= \gamma_1(x') \gamma_2(x_{2n}, \bar{x})
 \iint 
e^{i\big(\sigma(x_{2n}-\phi(x'))+\tau \cdot \bar{x}\big)} 
 d\sigma d\tau.
\end{equation}
The integral is understood in the sense of distributions.
To simplify the analysis, we decompose the integral by introducing dyadic decompositions, as per the notations established in \cite{MS}. 
Let $\zeta_0\in C^\infty_0(\R)$ be an even cut-off  so that
$\zeta_0(s)=1$ if $|s|\le 1/2$ and vanishes outside $(-1,1)$.
We let  $\zeta_1(s):=\zeta_0(s/2)-\zeta_1(s).$ To implement dyadic decompositions in $(\sigma, \tau)$ and subsequently in $\sigma$ when $|\sigma| < |\tau|$, we introduce the following partition of unity:
$$\beta_0+\sum_{k\ge 1}\big(\beta_{k,0}+\sum_{1\le l<k/3}\beta_{k,l}
+\widetilde \beta_{k}\big)=1,$$
where $\beta_{0}(\sigma,\tau)=\zeta_0(\sqrt{\sigma^2+|\tau|^2})$, and for $k\geq 1,~1\leq l<k/3$
\begin{align*}
\beta_{k,0}(\sigma,\tau)&=\zeta_1(2^{-k}\sqrt{\sigma^2+|\tau|^2})(1-\zeta_0(2^{-k}\sigma))
\\
\beta_{k,l}(\sigma,\tau)&=\zeta_1(2^{-k}\sqrt{\sigma^2+|\tau|^2})\zeta_1(2^{l-k}\sigma)
\\
\widetilde \beta_{k}(\sigma,\tau)&=\zeta_1(2^{-k}\sqrt{\sigma^2+|\tau|^2})
\zeta_0(2^{[k/3]-k-1}\sigma).
\end{align*}

Note that  for $k>0$ the function $\beta_{k,0} $ is supported
where $|\sigma| \approx 2^k$ and $|\tau|\lesssim 2^k$,
$\beta_{k,l} $ is supported
where $|\tau|\approx 2^k$ and $|\sigma|\approx 2^{k-l}$ and
$\widetilde \beta_k$ is supported where
$|\tau|\approx 2^k$ and $|\sigma|\lesssim  2^{2k/3}$. By applying this partition of unity to the integral in \eqref{diracFT}, we obtain 
$$\mu=
\mu^0+\sum_{k\ge 1}\big(\mu^{k,0}+\sum_{1\le l<k/3} \mu^{k,l}
+ \widetilde \mu^{k}\big)$$
where for $k\geq 1$ 
\begin{align*}
\mu^{0}(x)&= \gamma_1(x') \gamma_2(x_{2n}, \bar{x})
\iint 
e^{i\big(\sigma(x_{2n}-\phi(x'))+\tau \cdot \bar{x}\big)}  
\beta_0(\sigma,\tau)d\sigma d\tau,
\\
\mu^{k,l}(x)
&= \gamma_1(x') \gamma_2(x_{2n}, \bar{x})
\iint e^{i\big(\sigma(x_{2n}-\phi(x'))+\tau \cdot \bar{x}\big)} 
\beta_{k,l}(\sigma,\tau)d\sigma d\tau,\quad (0\le l<k/3)
\\
\widetilde \mu^{k}(x)&= \gamma_1(x') \gamma_2(x_{2n}, \bar{x}) \iint
e^{i\big(\sigma(x_{2n}-\phi(x'))+\tau \cdot \bar{x}\big)}   
\widetilde \beta_k(\sigma,\tau)d\sigma d\tau;
\end{align*}
Now  for $t>0$ we define the  $t$-dilate of any kernel $K$ as $$K_t(x)=t^{-Q} K(t^{-1}\underline{x}, t^{-2} \bar{x})\quad\quad (x=(\underline{x}, \bar{x})\in \Ha).$$ For $k\geq1,~ 0\leq l<k/3$, we consider the following maximal functions: 
\begin{align*}
    M_0f:=\sup_{t>0} |f* \mu^0_{t}|~,~ M_{k,l}f:=\sup_{t>0} |f* \mu^{k,l}_{t}|~, \widetilde{M}_{k}f:=\sup_{t>0}|f\ast \widetilde \mu^{k}_t|
\end{align*}
Therefore, it is easy to see that 
\begin{equation}
    \label{decomposeM}
    M_Sf\leq M_0f+\sum_{k\ge 1}\big(M_{k,0}f+\sum_{1\le l<k/3} M_{k,l}f
+ \widetilde M_kf\big).
\end{equation}
Note that $\mu^0$ is a bounded, compactly supported function whence it follows that 
the maximal function $M_{0}$ is pointwise dominated by the Hardy-Littlewood maximal function $\mathcal{M}$ and therefore it satisfies all the required estimates. Therefore, we need to prove vector-valued estimates for the operators $M_{k, l}$ and $\widetilde{M_{k}}$. We now record the following $L^2$-estimates of these maximal functions proved in \cite[Proposition 2.1]{MS}.
\begin{prop}
\label{PropMS}
 For $k\geq 1,~0\le l<k/3$ 
 \begin{align}
     \big\|M_{k,l}f\big\|_2\lesssim \sqrt{k}\, 2^{-k(2n-2)/2}\|f\|_2,~
\big\|\widetilde M_{k}f\big\|_2 
\lesssim \sqrt{k}\, 2^{-k(2n-2)/2} \|f\|_2.
 \end{align}  
\end{prop}

\section{Proof of main results}
\label{Main-Sec}
This section is dedicated to establishing the main results of this article. We begin by proving the boundedness of the vector-valued Nevo--Thangavelu spherical maximal operator.
\subsection{Vector-valued spherical maximal function}
The proof of Theorem~\ref{Sph-Main-Heisenber} relies fundamentally on a set of crucial vector-valued estimates. The first of these estimates follows as a straightforward consequence of Proposition ~\ref{PropMS}.
\begin{theorem}
\label{l2}
Let $(f_j)_{j\geq1}$ be a sequence of measurable functions defined in $\Ha$. If $\bigl(\sum_{j=1}^{\infty}|f_j(\cdot)|^2\bigr)^\frac{1}{2} \in L^2(\Ha)$, then
\[
\biggl\|\Bigl(\sum_{j=1}^{\infty}|M_{k, l}f_j(\cdot)|^2\Bigr)^\frac{1}{2}\biggr\|_{L^2(\Ha)}\leq C(n)\sqrt{k}\, 2^{-k(2n-2)/2}\biggl\|\Bigl(\sum_{j=1}^{\infty}|f_j(\cdot)|^2\Bigr)^\frac{1}{2}\biggr\|_{L^2(\Ha)}.
\]
Similar statemtent is true for $\widetilde M_k.$
\end{theorem}
\begin{proof}
Since we can commute the $L^2$ and $\ell^2$ norms, the proof simply follows from Proposition~\ref{PropMS}.     
\end{proof}

\begin{theorem}
\label{FS}
Let $(f_j)_{j\geq1}$ be a sequence of measurable functions defined on $\Ha$. If $\bigl(\sum_{j=1}^{\infty}|f_j(\cdot)|^q\bigr)^\frac{1}{q} \in L^p(\Ha)$, then we have the following vector-valued estimates
\begin{enumerate}[i)]
    \item 
    \label{item1}
\[
\biggl\|\Bigl(\sum_{j=1}^{\infty}|M_{k, l}f_j(\cdot)|^q\Bigr)^\frac{1}{q}\biggr\|_{L^p(\Ha)}\leq C(n, p, q)\,2^{2k}\biggl\|\Bigl(\sum_{j=1}^{\infty}|f_j(\cdot)|^q\Bigr)^\frac{1}{q}\biggr\|_{L^p(\Ha)},\]
for all $1<p, q<\infty.$
\item 
\label{item2} 
\[
\biggl\|\Bigl(\sum_{j=1}^{\infty}|\widetilde M_k \,f_j(\cdot)|^q\Bigr)^\frac{1}{q}\biggr\|_{L^p(\Ha)}\leq C(n, p, q)\,2^{2k}\biggl\|\Bigl(\sum_{j=1}^{\infty}|f_j(\cdot)|^q\Bigr)^\frac{1}{q}\biggr\|_{L^p(\Ha)},\]
for all $1<p, q<\infty.$
\end{enumerate}
\end{theorem}

\begin{proof}
We will only prove \eqref{item1} and the proof of \eqref{item2} follows analogously. Here onwards $\gamma(x)$ will denote $\gamma_1(x') \gamma_2(x_{2n}, \bar{x}).$ 
Let us recall the convoultion kernel $\mu^{k, l},$ for $0\le l<k/3,$
\begin{align}
\nonumber\mu^{k,l}(x)
&= \gamma(x)
\iint e^{i\big(\sigma(x_{2n}-\phi(x'))+\tau \cdot \bar{x}\big)} 
\beta_{k,l}(\sigma,\tau)d\sigma d\tau,\quad\\  \nonumber&=\gamma(x)
\iint e^{i\big(\sigma(x_{2n}-\phi(x'))+\tau \cdot \bar{x}\big)} 
\zeta_1(2^{-k}\sqrt{\sigma^2+|\tau|^2})\zeta_1(2^{l-k}\sigma)d\sigma d\tau,\quad.
\end{align}
A simple change of variable formula implies that

$$\mu^{k,l}(x)=\gamma(x) 2^{k}\, 2^{k-l}
\iint e^{i\big(2^{k-l}\sigma(x_{2n}-\phi(x'))+2^{k} \tau \cdot \bar{x}\big)} 
\zeta_1(\sqrt{\sigma^2+|\tau|^2})\zeta_1(\sigma)d\sigma d\tau\quad.$$

Applying integration by parts we get that
\begin{align*}
\big|\mu^{k,l}(x)\big| \leq C_{N} \frac{2^{km}}{(1+2^k|\bar{x}|)^N} \frac{2^{k-l}}{(1+2^{k-l}|x_{2n}-\phi(x')|)^N}\\
\lesssim  C_{N} \frac{2^{k}}{(1+|\bar{x}|)^N} \frac{2^{k}}{(1+|x_{2n}-\phi(x')|)^N}.  
\end{align*}

Also, considering the support properties of $\gamma,$ we observe that $|x'|\leq \rho$ and $\rho$ can be chosen suitably small to ensure that 
\begin{align}
\nonumber |x_{2n}-\phi(x')|&\simeq |x_{2n}-\sqrt{1-|x'|^2}|\\
\nonumber&\simeq |x_{2n}-(1-c|x'|^2)|\simeq |x'|^2,
\end{align}
considering the fact that $|x_{2n}-1|\lesssim \rho^2.$ Therefore, we obtain the pointwise bound
\begin{align}
\nonumber\big|\mu^{k,l}(x)\big| \leq C_{N} \frac{2^{k}}{(1+2^k|\bar{x}|)^N} \frac{2^{k-l}}{(1+2^{k-l}|x_{2n}-\phi(x')|)^N}\\
\lesssim  C_{N} \frac{2^{k}}{(1+|\bar{x}|)^N} \frac{2^{k}}{(1+|\underline{x}|)^N}. \label{ptwise} 
\end{align}
Using the estimate \eqref{ptwise} we obtain that 
\begin{align}
M_{k, l}f(x)=\sup_{t>0}|f* \mu^{k, l}_{t}(x)|(x)\lesssim 2^{2k} \mathcal{M}f(x).
\label{domHL}
\end{align}

At this point we invoke the Fefferman-Stein inequality for the Hardy-Littlwewood maximal function $\mathcal{M},$ that is, 
\[
\biggl\|\Bigl(\sum_{j=1}^{\infty}|\mathcal{M} f_j(\cdot)|^q\Bigr)^\frac{1}{q}\biggr\|_{L^p(\Ha)}\leq C(n, p, q)\, \biggl\|\Bigl(\sum_{j=1}^{\infty}|f_j(\cdot)|^q\Bigr)^\frac{1}{q}\biggr\|_{L^p(\Ha)}
\]
 holds true for all $1<p, q<\infty.$ Now the Fefferman-Stein inequality and \eqref{domHL} completes the proof.
\end{proof}

We now have the necessary tools to proceed with the proof of Theorem \ref{Sph-Main-Heisenber}.

\begin{proof}[Proof of Theorem~\ref{Sph-Main-Heisenber}]
Recall that
$$M_{S}f \leq M_0f+\sum_{k\ge 1}\big(M_{k,0}f+\sum_{1\le l<k/3} M_{k,l}f
+ \widetilde M_kf\big).$$
As pointed earlier, $M_0$ is pointwise dominated by the maximal function $\mathcal{M},$ and hence $M_{0}$ maps $L^p(\ell^q)$ to $L^p(\ell^q)$ for all $1<p, q<\infty.$ As a first step we prove $L^p(\ell^q)$ estimates for $M_{k, 0}, M_{k, l}$ and $\widetilde{M_{k}}$ and then summing over all frequencies $k\geq 1$ we will conclude the proof. We only provide the details for $M_{k, l}$ since the similar arguments are applicable for  $M_{k, 0}$ and $\widetilde{M_{k}}.$

Observe that for all $1<r,s<\infty$,  and $(f_j(\cdot))_{j\geq1} \in L^r(\Ha;\ell^s)$, we have
\begin{align}
\label{need1}
\biggl\|\Bigl(\sum_{j=1}^{\infty}{M_{k, l}} f_j(\cdot)|^s\Bigr)^\frac{1}{s}\biggr\|_{L^r(\Ha)}= \Bigl\| \mathcal{B}_{l}\bigl((f_j(\cdot))_{j\geq1}\bigr) \Bigr\|_{L^r(\Ha;\, \ell^s(L^\infty(\,(0, \infty)\,)))},
\end{align}
where 
$$\mathcal{B}_{k, l} :  L^r(\Ha;\ell^s) \to L^r(\Ha;\ell^s(L^\infty(\,(0, \infty)\,))),$$
defined by $\mathcal{B}_{l}((f_j(\cdot))_{j\geq1}) =\bigl(t \mapsto f_j*\mu^{k, l}_{t} \bigr)_{m\geq 1}.$
 Rewriting in this fashion, Proposition \ref{l2} implies that
 \begin{equation}\label{interpol1}\|\mathcal{B}_{k, l}\|_{L^2(\Ha;\ell^2) \to L^2(\Ha;\ell^2(L^\infty(\,(0, \infty))\,)))} \leq C(n
 )\, \sqrt{k}\, 2^{-k(2n-2)/2}\end{equation}
Similarly, Proposition \ref{FS} yields the following
\begin{equation}\label{interpol12} \|\mathcal{B}_{k, l}\|_{L^{r}(\mathbb{R}^n;\ell^{s}) \to L^{r}(\mathbb R^n;\ell^{s}(L^\infty(\,(0, \infty)\,)))} \leq C({n, r, s})\, 2^{2k},\end{equation}
for any $1 < r, s < + \infty$. Using complex interpolation between \eqref{interpol1} and \eqref{interpol12} we get 
\begin{align}\label{interpol3}
\nonumber \|\mathcal{B}_{k, l}\|_{L^p(\mathbb{R}^n;\ell^q) \to L^p(\mathbb R^n;\ell^q(L^\infty(\,(0, \infty)\,))) } &\leq C(n, p, q) \sqrt{k}\,2^{k\big((-\frac{(2n-2)}{2})\theta +2(1-\theta)\big)}\\
&= C(n, p, q)\, \sqrt{k}\,2^{k\gamma(n,\theta)},
\end{align}
where $$\gamma(n, \theta):=-\frac{(2n-2)}{2}\theta +2(1-\theta),$$ for any $\theta\in (0, 1),$ and
\begin{align}\frac1p = \frac\theta 2 + (1 - \theta) \frac1{r},\,\,\,\,\frac1q = \frac\theta 2  + (1 - \theta) \frac1{s}\label{relation}.\end{align}

Upon simplification, $\gamma(n, \theta)<0$ if and only if 
\begin{align*}
&2<\theta \bigg(2+\frac{2n-2}{2}\bigg)\\
&\iff \theta>\frac{4}{2n+2}=\frac{4}{Q}.
\end{align*}
 For any $\frac{Q}{Q-2}<p, q<\frac{Q}{2}$ we can find $\theta$ such that 
 \begin{align}
 \nonumber \frac{2}{Q}<\frac{\theta}{2}<\frac{1}{p}, \frac{1}{q}<1-\frac{\theta}{2}<\frac{Q-2}{Q} 
 \end{align}
 Therefore, we proved that for $\frac{Q}{Q-2}<p, q<\frac{Q}{2}$ there exists a $\gamma:=\gamma(n, \theta)<0$ such that 
\begin{align}
\label{FS1}
\biggl\|\Bigl(\sum_{j=1}^{\infty}{M_{k, l}} f_j(\cdot)|^q\Bigr)^\frac{1}{q}\biggr\|_{L^p(\Ha)}\leq C({n, p, q}) 2^{k\gamma} \sqrt{k} \, \biggl\|\Bigl(\sum_{j=1}^{\infty} |f_j(\cdot)|^q\Bigr)^\frac{1}{q}\biggr\|_{L^p(\Ha)}.    
\end{align} 
A similar argument will imply that
\begin{align}
\label{FS2}
\biggl\|\Bigl(\sum_{j=1}^{\infty}{M_{k, 0}} f_j(\cdot)|^q\Bigr)^\frac{1}{q}\biggr\|_{L^p(\Ha)}\leq C({n, p, q})\, 2^{k\gamma} \sqrt{k} \, \biggl\|\Bigl(\sum_{j=1}^{\infty} |f_j(\cdot)|^q\Bigr)^\frac{1}{q}\biggr\|_{L^p(\Ha)},    
\end{align} 
and 
\begin{align}
\label{FS3}
\biggl\|\Bigl(\sum_{j=1}^{\infty}\widetilde{M_{k}}f_j(\cdot)|^q\Bigr)^\frac{1}{q}\biggr\|_{L^p(\Ha)}\leq C({n, p, q})\, 2^{k\gamma} \sqrt{k} \, \biggl\|\Bigl(\sum_{j=1}^{\infty} |f_j(\cdot)|^q\Bigr)^\frac{1}{q}\biggr\|_{L^p(\Ha)}.    
\end{align} 
Combining the above estimates we obtain that

\begin{align}
\nonumber&\biggl\|\Bigl(\sum_{j=1}^{\infty}|M_{S}f_j(\cdot)|^q\Bigr)^\frac{1}{q}\biggr\|_{L^p(\Ha)}\\
\nonumber&\leq \sum_{k\geq 1} \left(\sum_{1\leq l\leq \frac{k}{3}}\biggl\|\Bigl(\sum_{j=1}^{\infty}{M_{k, l}} f_j(\cdot)|^q\Bigr)^\frac{1}{q}\biggr\|_{L^p(\Ha)}+\biggl\|\Bigl(\sum_{j=1}^{\infty}{M_{k, 0}} f_j(\cdot)|^q\Bigr)^\frac{1}{q}\biggr\|_{L^p(\Ha)}\right.\\
\nonumber &\left.+\biggl\|\Bigl(\sum_{j=1}^{\infty}\widetilde{M_{k}}f_j(\cdot)|^q\Bigr)^\frac{1}{q}\biggr\|_{L^p(\Ha)}\right)\\
\nonumber&\leq C(n, p, q)  \biggl\|\Bigl(\sum_{j=1}^{\infty}|f_j(\cdot)|^q\Bigr)^\frac{1}{q}\biggr\|_{L^p(\Ha)} \sum_{k\geq 1}\left(2^{k\gamma}k^{5/2}\right)\\
&\nonumber\leq C(n, p,q) \biggl\|\Bigl(\sum_{j=1}^{\infty}|f_j(\cdot)|^q\Bigr)^\frac{1}{q}\biggr\|_{L^p(\Ha)},   
\end{align}
as $\gamma<0$ and which in turn is true since $\frac{Q}{Q-2}<p, q<\frac{Q}{2}.$

\end{proof}

\subsection{Proof of Theorem~\ref{Main-thm}}
In this section we provide the proof of Theorem~\ref{Main-thm}. We start with the following essential lemma which articulates that the Hardy-Littlewood maximal function over Kor\'anyi averages can be dominated pointwise by the composition of an one-dimensional Hardy-Littlewood maximal operator and the Nevo--Thangavelu maximal operator $M_{S}$. More presicely, we have the following: 
\begin{lemma}
\label{hlmax-ptest}
    For $f\in C(\Ha),$ the following pointwise estimate holds: 
    $$\mathcal{M}f(x)\leq M_{\mathbb{R}}(M_Sf)(x),$$
    for $x\in \Ha.$ 
\end{lemma}
\begin{proof}
     In view of the left-invariance of $\mathcal{M}$, it is enough to prove that 
     \begin{equation}
     \label{dimestclaim}
         \mathcal{M}f(0)\leq M_{\mathbb{R}}(M_Sf)(0).
     \end{equation}
    In order to do so, for any $r>0$, we note that 
     \begin{align}
     \label{dimestlemeq1}
         \int_{B(0,r)}|f(x)|dx= \int_{|\bar{x}|\leq r^2}\int_{\Sigma^r_{\bar{x}}}|f(\underline{x}, \bar{x})| d\underline{x}~d\bar{x},
     \end{align}
     where $\Sigma^r_z$ is defined by 
     $$\Sigma^r_{\bar{x}}=\{\underline{x}\in \mathbb{R}^{2n}:  |\underline{x}|^4+|\bar{x}|^2\leq r^4\}.$$
     Using polar decomposition in the integration over $\Sigma^r_{\bar{x}}$ for fixed $\bar{x}$, we have, 
     \begin{align*}
         \int_{\Sigma^r_{\bar{x}}}|f(\underline{x}, \bar{x})| d\underline{x}&=\int_{\{s>0: s^4+|\bar{x}|^2\leq r^4\}}\int_{S^{2n-1}} |f(s\omega, \bar{x})|s^{2n-1}d\mu(\omega)~ds\\
         &\leq \mu(S^{2n-1})\int_{\{s>0: s^4+|\bar{x}|^2\leq r^4\}} M_Sf(0, \bar{x}) s^{2n-1}~ds,
     \end{align*}
     which, taking into account \eqref{dimestlemeq1}, implies 
     \begin{align*}
         &\frac{1}{B(0,r)|} \int_{B(0,r)}|f(x)|dx\\ &\leq \frac{1}{B(0,r)|}\int_{|\bar{x}|\leq r^2}\left(\mu(S^{2n-1})\int_{\{s>0: s^4+|\bar{x}|^2\leq r^4\}}s^{2n-1}ds\right) M_Sf(0, \bar{x})~d\bar{x}.
     \end{align*}
     But notice that the quantity in the bracket in the above integral is nothing but the $2n$-dimensional Lebesgue measure $|\Sigma^r_{\bar{x}}|$ of $\Sigma^r_{\bar{x}}.$ Therefore, writing for every $r>0$, $$\varphi_r(\bar{x})=\frac{1}{|B(0, r)|}  |\Sigma^r_{\bar{x}}|\chi_{B_{\R}(0,r^2)}(\bar{x})\quad\quad (\bar{x}\in \mathbb{R}),$$ we arrive at 
     \begin{equation}
     \label{dimestlemeq2}
         \frac{1}{|B(0, r)|} \int_{B(0, r)}|f(x)|dx\leq \int_{\R} \varphi_r(\bar{x}) M_Sf(0, \bar{x})~d\bar{x}.
     \end{equation}
    Now notice first that $\int_{\R}\varphi_r(\bar{x})d\bar{x}=1$. Indeed, this is easy to see from the observation that 
    $$\int_{|\bar{x}|\leq r^2} |\Sigma^r_{\bar{x}}|\,d\bar{x}= |B(0, r)|.$$
   Realizing now that for each $\bar{x}$ fixed,  $\Sigma^r_{\bar{x}}$ is simply a standard $2n$-dimensional Euclidean ball of radius $(r^4-|\bar{x}|^2)^{\frac14}$, we have $|\Sigma^r_{\bar{x}}|=c_n (r^4-|\bar{x}|^2)^{\frac{n}2}$, proving directly that $\varphi_r$ is radially decreasing. Hence, using \cite[Theorem 2.1.10]{Classical-Gra} from \eqref{dimestlemeq2} we obtain 
   $$\frac{1}{|B(0, r)|} \int_{B(0, r)}|f(x)|dx\leq M_{\R}(M_Sf)(0)$$
   whence \eqref{dimestclaim} follows.    
\end{proof}

\begin{proof}[Proof of Theorem~\ref{Main-thm}]
    For any $m\leq n$, we view $\Ha$ as $\mathbb{H}^{m}\times \R^{2(n-m)}$ and write $x\in \Ha$ as $x=({x'}, {x''}),$ where $$x'=(\underline{x'}, \overline{x'})=(x_{1}, \ldots, x_{m},0, \ldots, 0,x_{n+1}, \ldots, x_{n+m}, 0,\ldots, 0, \overline{x})\in \mathbb{H}^{m}.$$
    Moreover, we view the normalised measure on $S^{2m-1}$ as a normalised measure supported on $(2m-1)$-dimensional sphere in the hyperplane $$\{x=({x'}, {x''})\in \mathbb{H}^{m}\times \R^{2(n-m)}: \overline{x'}=0=x''\}$$
    and call it $\tilde{\mu}.$ We then consider the measure $\nu_r$ defined by 
    $$\nu_r(E)=\int_{U(n)} \tilde{\mu}_r(\sigma \cdot E)\,d\sigma,$$
    where $d\sigma$ denote the normalised Haar measure on $U(n).$ Clearly $\nu_r$ defines a normalised surface measure on the sphere of radius $r$ in $\mathbb{R}^{2n}$ whence $\nu_r=\mu_r$ by uniqueness. Therefore, we have 
    \begin{align*}
        f\ast \mu_r(x)&=f\ast \nu_r(x)= \int_{U(n)}R_{\sigma}f \ast \tilde{\mu}_r(\sigma^{-1}\cdot x)\,d\sigma\\
        \nonumber &= \int_{U(n)}R_{\sigma^{-1}}\left(R_{\sigma}f \ast \tilde{\mu}_r\right)(x)\,d\sigma
    \end{align*}
    leading to 
    \begin{align}
    \label{dombyavgofrot}
        M_Sf(x)\leq \int_{U(n)} \sup_{r>0}|R_{\sigma^{-1}}\left(R_{\sigma}f \ast \tilde{\mu}_r\right)( x)|\, d\sigma.
    \end{align}
    where recall that $R_{\sigma}f$ stand for the rotation of $f$ in the horizontal hyperplane by $\sigma$, viz. $R_{\sigma}f(x):=f(\sigma\cdot \underline{x},\overline{x})$ for $x=(\underline{x},\overline{x})\in\mathbb{H}^n.$  For any $\sigma\in U(n)$, writing $$\widetilde{M}^{\sigma}_Sf(x):= \sup_{r>0}|R_{\sigma^{-1}}\left(R_{\sigma}f \ast \tilde{\mu}_r\right)( x)| \quad\quad (x\in \Ha),$$ we next claim that whenever $\frac{m+1}{m}<p,q< {m+1},$ we have
    \begin{equation}
    \label{dombydimless}
\biggl\|\Bigl(\sum_{j=1}^{\infty}|\widetilde{M}^{\sigma}_{S}f_j(\cdot)|^q\Bigr)^\frac{1}{q}\biggr\|_{L^p(\Ha)}\leq C({m, p, q})\,\biggl\|\Bigl(\sum_{j=1}^{\infty}|f_j(\cdot)|^q\Bigr)^\frac{1}{q}\biggr\|_{L^p(\Ha)}
    \end{equation}
 for any  sequence of measurable functions $(f_j)_{j\geq1}$ such that  $\bigl(\sum_{j=1}^{\infty}|f_j(\cdot)|^q\bigr)^\frac{1}{q} \in L^p(\Ha)$, where the constant $C({m,p,q})$ is independent of $n$ and $\sigma.$ To substantiate this claim, using the fact that the Lebesgue measure is rotation invariant, we note that 
 \begin{align*}  \biggl\|\Bigl(\sum_{j=1}^{\infty}|\widetilde{M}^{\sigma}_{S}f_j(\cdot)|^q\Bigr)^\frac{1}{q}\biggr\|_{L^p(\Ha)} = \biggl\|\Bigl(\sum_{j=1}^{\infty}\big(\sup_{r>0}|R_{\sigma}f_j\ast \widetilde{\mu}_r(\cdot)|\big)^q\Bigr)^\frac{1}{q}\biggr\|_{L^p(\Ha)}.
 \end{align*}
But, in view of the identification $\Ha=\mathbb{H}^{m}\times \R^{2(n-m)}$ as mentioned in the beginning of the proof, the $p$th power of the last expression takes the form 
\begin{align}
\label{RS}
    \int_{\mathbb{H}^{m}} \int_{\R^{2(n-m)}} \Bigl(\sum_{j=1}^{\infty}\big(\sup_{r>0}|R_{\sigma}f_j\ast \widetilde{\mu}_r(x',x'')|\big)^q\Bigr)^\frac{p}{q} dx''dx'.
\end{align}
But then, recalling the Heisenberg group law, it is easy to see that 
$$f\ast \widetilde{\mu}_r(x)= f_{x''}\ast_{\mathbb{H}^{m}}\widetilde{\mu}_r(x')$$
where $f(x)=f(x',x'')=f_{x''}(x'),$ and $\widetilde{\mu}_r$ has been viewed as the normalised measure on the $(2m-1)$-dimensional complex sphere of radius $r$ in $\mathbb{H}^{m}.$ Using this observation and Fubini,  we rewrite the expression \eqref{RS} as 
\begin{align*}
    \int_{\R^{2(n-m)}} \left( \int_{\mathbb{H}^{m}} \Bigg(\sum_{j=1}^\infty \big(M^{m}_S((R_{\sigma}f_j)_{x''})(x')\big)^q\Bigg)^{\frac{p}q}dx'\right)dx''
\end{align*}
where $M^{m}_S$ stands for the spherical maximal operator in $\mathbb{H}^{m}$, in which,  the superscript is used to emphasize the dimension. Now whenever $\frac{m+1}{m}<p,q< {m+1}$, by Theorem \ref{Sph-Main-Heisenber}, we have 
\begin{align*}
    & \int_{\R^{2(n-m)}} \left( \int_{\mathbb{H}^{m}} \Bigg(\sum_{j=1}^\infty \big(M^{m}_S((R_{\sigma}f_j)_{x''})(x')\big)^q\Bigg)^{\frac{p}q}dx'\right)dx''\\
    & \leq C({m,p,q})^p \int_{\R^{2(n-m)}} \left( \int_{\mathbb{H}^{m}} \Bigg(\sum_{j=1}^\infty |(R_{\sigma}f_j)_{x''}(x')|^q\Bigg)^{\frac{p}q}dx'\right)dx''\\
    &= C({m,p,q})^p \int_{\R^{2(n-m)}} \int_{\mathbb{H}^{m}} \Bigg(\sum_{j=1}^\infty |(R_{\sigma}f_j)(x',x'')|^q\Bigg)^{\frac{p}q}dx'dx''\\
    &=C({m,p,q})^p\biggl\|\Bigl(\sum_{j=1}^{\infty}|f_j(\cdot)|^q\Bigr)^\frac{1}{q}\biggr\|^p_{L^p(\Ha)},
\end{align*}
where in the last equality we have used the rotation invariance of the Haar measure. This proves the claim \eqref{dombydimless}.

We now complete the proof by appealing to the observation \eqref{dombydimless}, and Lemma \ref{hlmax-ptest}. First, let $1<p,q<\infty$, and $(f_j)_{j\geq1}$ be a sequence of measurable functions on $\Ha$ such that  $\bigl(\sum_{j=1}^{\infty}|f_j(\cdot)|^q\bigr)^\frac{1}{q} \in L^p(\Ha).$ 
Now if $n\leq \max\{\frac{1}{p-1}, \frac{1}{q-1}\}$ or $n\leq \max \{p-1, q-1\}$, then the result follows from the standard Fefferman-Stein inequality for Hardy-Littlewood maximal operator, that is from Theorem~\ref{ThmFS}, which is known for more general context of spaces of homogeneous type. 
Now to complete the proof, it remains to consider the case when  $\frac{n+1}{n}<p,q< n+1$ or in other words, $\frac{Q}{Q-2}<p,q<\frac{Q}2.$  In this case, in view of the Lemma \ref{hlmax-ptest},   using the Fefferman-Stein inequality for one-dimensional maximal operator $M_{\R}$, we obtain 
\begin{align*}
\biggl\|\Bigl(\sum_{j=1}^{\infty}|\mathcal{M}f_j(\cdot)|^q\Bigr)^\frac{1}{q}\biggr\|_{L^p(\Ha)}\leq \biggl\|\Bigl(\sum_{j=1}^{\infty}|M_Sf_j(\cdot)|^q\Bigr)^\frac{1}{q}\biggr\|_{L^p(\Ha)}.
\end{align*} 
We now set $$n_0:= \left\lfloor\max\bigg\{ \frac{1}{p-1}, \frac{1}{q-1}, p-1, q-1\bigg\}\right\rfloor+1.$$ Then $n_0<n$ and applying \eqref{dombyavgofrot} with $m=n_0$, we have
$$\biggl\|\Bigl(\sum_{j=1}^{\infty}|M_Sf_j(\cdot)|^q\Bigr)^\frac{1}{q}\biggr\|_{L^p(\Ha)}\leq \biggl\|\Bigl(\sum_{j=1}^{\infty}|\widetilde{M}^\sigma_Sf_j(\cdot)|^q\Bigr)^\frac{1}{q}\biggr\|_{L^p(\Ha)}.$$
But since by definition, $\frac{n_0+1}{n_0}<p,q<n_0+1$, by using \eqref{dombydimless}, the above inequality yields 
$$\biggl\|\Bigl(\sum_{j=1}^{\infty}|\mathcal{M}f_j(\cdot)|^q\Bigr)^\frac{1}{q}\biggr\|_{L^p(\Ha)}\leq C({n_0,p,q})\biggl\|\Bigl(\sum_{j=1}^{\infty}|f_j(\cdot)|^q\Bigr)^\frac{1}{q}\biggr\|_{L^p(\Ha)}.$$
Finally, as $n_0$ only depends on $p$ and $q$, this completes the proof of the theorem.
\end{proof}

\subsection{UMD Banach lattice-valued maximal function}
We begin by recalling some relevant properties of the heat semi-group associated with the sublaplacian $\mathcal{L}$ on $\Ha.$  It is well-known that the fundamental solution of the corresponding heat equation $$\partial_t u(x,t)+ \mathcal{L}u(x,t)=0,\quad\,\,((x,t)\in \Ha\times \R^+)$$
is given by the heat kernel, $q_t$ which has the following properties (see e.g., \cite{FolS})
\begin{enumerate}[(I)]
    \item The function $(x,t)\mapsto q_t(x)$ is smooth on $\Ha\times \R^+$ satisfying \\$q_t(x)=t^{-Q/2} q_1(\delta_{t^{-1/2}}x).$ Moreover, $\|q_t\|_1=1.$
    \item There exist positive constants $C, c_1$, and $c_2$ such that 
    \begin{equation}
        \label{heatest}
        \frac1C~t^{-\frac{Q}{2}}~e^{-\frac{|x|^2}{c_1t}}\leq q_t(x)\leq C ~t^{-\frac{Q}{2}}~e^{-\frac{|x|^2}{c_2t}}\quad\quad ((x,t)\in \Ha\times\R^+).
    \end{equation}
    \end{enumerate}
    It then follows that the associated semi-group defined by $$e^{-t\mathcal{L}}f:=f\ast q_t \quad\quad (f\in L^p(\Ha), 1\leq p\leq \infty)$$
   is a regular contraction on $L^p(\Ha)$, and self-adjoint on $L^2(\Ha).$ Therefore, by Stein \cite[Theorem 1, Section 2, Chapter III]{SteinLP}, it follows that for $1<p<\infty$, the map $t\mapsto e^{-t\mathcal{L}}$ extends to an anlytic $L^p$ operator valued map defined in the sector $$S_p:=\bigg\{z\in\mathbb{C}: z\neq 0,~|\arg (z)|<\frac\pi2 \bigg(1-\bigg|\frac2p-1\bigg|\bigg) \bigg\},$$
   establishing that $\{e^{-t\mathcal{L}}\}_{t>0}$ is an analytic semigroup on $L^p(\Ha).$ Consequently, given an UMD Banach lattice $E$ on a measure space $(\Omega, \nu)$, in view of  \cite[Theorem 2]{X},  the associated sectorial heat maximal function defined by 
   $$M_{\mathcal{L}}f(x, \omega):=\sup_{z\in S_p}|e^{-z\mathcal{L}}f(\cdot,\omega)(x)|,\quad\quad ((x,\omega)\in \Ha\times \Omega),$$
   satisfies 
   \begin{equation}
   \label{heatmaxumd}
       \|M_{\mathcal{L}}f\|_{L^p(\Ha; E)}\lesssim \|f\|_{L^p(\Ha; E)}\quad\quad (f\in L^p(\Ha; E)).
   \end{equation}
   Now using the heat kernel estimate \eqref{heatest}, we notice that 
   $$t^{-Q/2}\chi_{B(0, t^{-1/2})}(x)\lesssim t^{-\frac{Q}{2}}~e^{-\frac{|x|^2}{c_1t}}\leq q_t(x),$$
   which proves the following pointwise estimate 
   $$\mathcal{M}f(x)\lesssim \sup_{t>0}|f\ast q_t(x)|\quad\quad (x\in \Ha).$$
   This, together with \eqref{heatmaxumd}, then proves the following: 
   \begin{prop}
   \label{withdimumdbdd}
       Let $1<p<\infty$, and $E$ be an UMD Banach lattice on a measure space $(\Omega, \nu)$. Then there exists a constant $C(n)$ depending on $n$ such that 
       $$\|\mathcal{M}f\|_{L^p(\Ha; E)}\leq C(n) \|f\|_{L^p(\Ha; E)}\quad\quad (f\in L^p(\Ha; E)).$$
   \end{prop}
    We now turn to the proof of the second main result. 
    \begin{proof}[Proof of Theorem \ref{main-thm2}]
        Fix $1<p<\infty$, and an UMD Banach lattice $Z.$ First, we observe that it is enough to prove that there exists $n_0$ (sufficiently large) depending only on $p$, and $Z$ such that for all $n\geq  n_0$ we have  
        \begin{equation}
            \label{sph-umdest}
            \|M_Sf\|_{L^p(\Ha; Z)}\leq C(n,p,Z) \|f\|_{L^p(\Ha; Z)}\quad\quad (f\in L^p(\Ha; Z)).
        \end{equation}
        Indeed, assuming \eqref{sph-umdest}, the proof can be completed by adapting the arguments appropriately as in the proof of Theorem \ref{Main-thm}. In fact,  first, for any $\sigma\in U(n)$, by repeating the same arguments and replacing $l^q$ with $Z$, we observe that for $m\geq n_0$
        \begin{align*}
            \|\widetilde{M}^{\sigma}_Sf\|_{L^p(\Ha; Z)}\leq C(m,p,Z) \|f\|_{L^p(\Ha; Z)}\quad\quad (f\in L^p(\Ha; Z))
        \end{align*}
        and then for any $n>n_0$, we get 
        \begin{align*}
          \|M_Sf\|_{L^p(\Ha; Z)}\leq   \|\widetilde{M}^{\sigma}_Sf\|_{L^p(\Ha; Z)}\leq C(n_0,p,Z) \|f\|_{L^p(\Ha; Z)}\quad\quad (f\in L^p(\Ha; Z))
        \end{align*}
        Finally, using Lemma \ref{hlmax-ptest} along with the boundedness of UMD lattice $Z$-valued Euclidean maximal operator $M_{\R}$, we thus obtain 
        $$ \|\mathcal{M}f\|_{L^p(\Ha; Z)}\leq C(n_0,p,Z) \|f\|_{L^p(\Ha; Z)}$$
        for any $n\geq n_0.$ Now for $n<n_0$, we apply Proposition \ref{withdimumdbdd} to establish the desired dimension-independent bound. 
        
        Therefore, to complete the proof of the theorem, it remains to verify the claim \eqref{sph-umdest}. To do this, we proceed as in the proof of Theorem \ref{Sph-Main-Heisenber}. 
        For $x\in \Ha$ and $\omega\in\Omega$, as in \eqref{decomposeM}, we decompose
        \begin{align}
        \label{Msdecomp}
            M_Sf(x,\omega)\leq M_0f(x,\omega)+\sum_{k\ge 1}\big(M_{k,0}f(x,\omega)+\sum_{1\le l<k/3} M_{k,l}f
(x,\omega)+ \widetilde M_kf(x,\omega)\big).
        \end{align}
          Here, as mentioned earlier, we have 
          $$M_0f(x,\omega)\leq C(n) \mathcal{M}f(x,\omega)$$
          which then in view of Proposition \ref{withdimumdbdd} implies 
          \begin{equation}
              \label{umd-M0est}
              \|M_0f\|_{L^p(\Ha; Z)}\leq C(n,p,Z) \|f\|_{L^p(\Ha; Z)}.
          \end{equation}
         To estimate the remaining terms in the decomposition \eqref{Msdecomp}, we utilize the complex interpolation technique for Banach spaces.  We first record that there exists an UMD Banach lattice $X$ on $\Omega$ such that $Z$ can be obtained as the complex interpolation method 
         $[X, L^2(\Omega)]_{\theta}=Z$ with $\theta\in (0,1).$
         Consider the operators 
         $$\mathcal{B}_{k, l} :  L^r(\Ha; E) \to L^r(\Ha; E(L^\infty(\,(0, \infty)\,))),$$
defined by $$\mathcal{B}_{k,l}f(x,\omega,t) = f\ast \mu^{k, l}_{t}(x,\omega) \quad\quad ((x,\omega,t)\in \Ha\times \Omega\times (0,\infty)).$$
Consequently, $$\| M_{k,l}f\|_{L^p(\Ha; E)}=\|\mathcal{B}_{k,l}f\|_{L^p(\Ha; E(L^\infty(\,(0, \infty)\,)))}.$$
Now as $L^2(\Omega)$-norm and $L^2(\Ha)$-norm commutes, in view of Proposition \ref{PropMS}, and Theorem \ref{l2}, it follows that 
\begin{align}
\label{bkl-L2-umd-est}
  \|\mathcal{B}_{k,l}f\|_{L^2(\Ha; L^2(\Omega)(L^\infty(\,(0, \infty)\,)))}  \leq C(n) \sqrt{k}\,2^{-k(2n-2)/2}\|f\|_{L^2(\Ha; L^2(\Omega))}.
\end{align}
Also, using \eqref{ptwise}, we have 
\begin{align*}
    \| \mathcal{B}_{k,l}f(x,\omega,\cdot)\|_{L^{\infty}(0,\infty)}\leq C(n)\, 2^{2k} \mathcal{M}f(x,\omega)
\end{align*}
which then by Proposition \ref{withdimumdbdd} yields that
\begin{align}
\label{bkl-Lr-umd-est}
    \|\mathcal{B}_{k,l}f\|_{L^r(\Ha; E(L^\infty(\,(0, \infty)\,)))}\leq C(n, r, E)\, 2^{2k} \|f\|_{L^r(\Ha; E)}
\end{align}
for any $1<r<\infty.$ 

Using complex interpolation method $[X(L^\infty), L^2(\Omega)(L^\infty)]_{\theta}\hookrightarrow [X, L^2(\Omega)]_{\theta}(L^\infty)= Z(L^\infty)$, in view of \eqref{bkl-L2-umd-est}, and \eqref{bkl-Lr-umd-est}, we have 
\begin{equation}
    \label{bkl-lp-umdZ-est}
     \|\mathcal{B}_{k,l}f\|_{L^p(\Ha; Z(L^\infty(\,(0, \infty)\,)))}\leq C(n, r, Z)\sqrt{k}\,2^{k \gamma(n, \theta)}\, \|f\|_{L^r(\Ha; Z)},
\end{equation}
where $$\gamma(n, \theta):=-\frac{(2n-2)}{2}\theta +2(1-\theta)=2-\theta \frac{2n+2}{2}=2-\theta\frac{Q}{2} ,$$ for  any $\theta\in (0, 1),$ with an appropriate choice of $r\in (1,\infty)$ so that  
$$\frac1p = \frac\theta 2 + (1 - \theta) \frac1{r}= \frac1r-\theta\left(\frac1r-\frac12\right).$$
Now choose $\theta$ sufficiently small, i.e., $\theta<4/Q$ so that $\gamma(\theta, n)<0. $ Therefore, setting  $n_0:= \max \{ 2, \lfloor 2/\theta\rfloor\}$, we see that for $n\geq n_0$
\begin{align*}
     \|M_{k,l}f\|_{L^p(\Ha; Z(L^\infty(\,(0, \infty)\,)))}\leq C(n, r, Z)\sqrt{k}\,2^{k \gamma(n, \theta)} \|f\|_{L^r(\Ha; Z)}
\end{align*}
which is true for all $l\in [0,k/3),~k\geq 1.$
Moreover, a similar argument as above provides for any $k\geq 1$
\begin{align*}
     \|\widetilde{M}_{k}f\|_{L^p(\Ha; Z(L^\infty(\,(0, \infty)\,)))}\leq C(n, r, Z)\sqrt{k}\,2^{k \gamma(n, \theta)} \|f\|_{L^r(\Ha; Z)}.
\end{align*}
Finally, adding all above along with \eqref{umd-M0est}, considering \eqref{Msdecomp}, we obtain 
\begin{align*}
     \|M_Sf\|_{L^p(\Ha; Z)}\leq C(n,p,Z) \|f\|_{L^p(\Ha; Z)}
\end{align*}
This verifies the claim \eqref{sph-umdest}, thereby completing the proof of the theorem.
    \end{proof}

\subsection*{Concluding remarks} As a concluding remark we point out that our main dimension-free vector valued estimate is also true for the Hardy-Littlewood maximal function with respect to averages over balls in Carnot-Carath\'eodory metric on the Heisenberg group. To illustrate, recall that a curve $\gamma:[a, b]\to \Ha$ is called a \textit{horizontal curve} joining the two points $x, y\in \Ha$ if $\gamma(a)=x, \gamma(b)=y$ and $\gamma'(s)\in \text{span}\{X_{1}, \cdots, X_{n}, X_{n+1},\cdots, X_{2n}\}$ for all $s\in [a, b].$ Further, the Carnot-Carath\'eodory distance between $x$ and $y$ is defined as 
\[d_{CC}(x, y):=\inf_{\gamma}\int_{a}^{b}\|\gamma'(s)\|\, ds,\]
where the infimum is being taken over all \textit{horizontal curves} joining $x$ and $y.$ The associated Hardy-Littlewood maximal function is defined as 

\[\mathcal{M}_{CC}f(x)=\sup_{r>0}\frac{1}{|B_{CC}(x, r)|}\int_{B_{CC}(x, r)}|f(y)|\, dy,\]
where we denote $B_{CC}(x, r):=\{y: d_{CC}(x, y)<r\}$ and $|B_{CC}(x, r)|$ is the Lebesgue measure of the Carnot-Carath\'edory ball. Arguing exactly as in the proof of Theorem~\ref{Main-thm} we have the following dimension free estimates for the Hardy-Littlewood maximal function over Carnot-Carath\'eodory balls:
\begin{theorem}
\label{Main-thm-CC}
Let $1<p, q<\infty.$ Then there exist a constant $C(p, q),$ independent of dimension, such that
\[
\biggl\|\Bigl(\sum_{j=1}^{\infty}|\mathcal{M}_{CC}f_j(\cdot)|^q\Bigr)^\frac{1}{q}\biggr\|_{L^p(\Ha)}\leq C(p, q)\, \biggl\|\Bigl(\sum_{j=1}^{\infty}|f_j(\cdot)|^q\Bigr)^\frac{1}{q}\biggr\|_{L^p(\Ha)},
\]
for any  sequence $(f_j)_{j\geq 1}$  of measurable functions defined on $\mathbb{H}^n$ such that $\left(\sum_{j=1}^\infty|f_j(\cdot)|^q\right)^{\frac1q}\in L^p(\mathbb{H}^n).$

Moreover, arguing as in Theorem~\ref{main-thm2}, if $Z$ is a UMD Banach lattice, then there exists a constant $C(p, Z),$ independent of $n,$ such that
$$\bigg\| \big\|\mathcal{M}_{CC}(f)(\cdot, z)(x)\big\|_{Z}\bigg\|_{L^p(\Ha)}\leq C(p, Z)\, \|f\|_{L^p(\Ha,\, Z)}.$$
\end{theorem}

\end{document}